\documentclass{amsart}

\usepackage{amssymb,amsmath,latexsym}

\newtheorem{theorem}{Theorem}[section]
\newtheorem{lemma}[theorem]{Lemma}

\newtheorem{definition}[theorem]{Definition}
\newtheorem{proposition}[theorem]{Proposition}
\newtheorem{example}[theorem]{Example}

\newcommand\ds{\displaystyle}
\newcommand{\Hil}{\mathcal H}
\newcommand{\BH}{{\mathbf B}(\Hil)}

\newcommand{\C}{\mathbb C}
\newcommand{\D}{\mathbb D}
\newcommand{\N}{\mathbb N}

\newcommand{\0}{\{ 0 \}}
\newcommand{\calL}{\mathcal L}
\newcommand{\calM}{\mathcal M}
\newcommand{\calN}{\mathcal N}
\newcommand{\HC}{\operatorname{HC}}

\renewcommand{\span}{\operatorname{span}}

\newcommand{\ran}{\operatorname{ran}}

\newcommand{\Orb}{\operatorname{Orb}}

\renewcommand{\emptyset}{\varnothing}
\def\conj#1{\overline{#1}}
\newcommand{\restricted}{\bigm|}

\title{Subspace Hypercyclicity}

\author{Blair F.~Madore}
\address{Department of Mathematics, State University of New York
  College at Potsdam, Potsdam, NY 13676-2294, U.S.A.}
\email{bfmadore@potsdam.edu}

\author{Rub\'en A.~Mart\'inez-Avenda\~no}
\address{Centro de Investigaci\'on en Matem\'aticas,  Universidad
  Aut\'onoma del Estado de Hidalgo, Ciudad Universitaria,
  Carr.~Pachuca--Tulancingo, km 4.5,  Pachuca, Hidalgo, 42184, Mexico}
\email{rubenma@uaeh.edu.mx}



\thanks{The authors would like to thank the State University of New
  York College at Potsdam and Universidad Aut\'onoma del
  Estado de Hidalgo for making this collaboration possible.}

\begin{document}

\begin{abstract}
A bounded linear operator $T$ on Hilbert space is subspace-hypercyclic
for a subspace $\calM$ if there exists a vector whose orbit under
$T$ intersects the subspace in a relatively dense set. We construct
examples to show that subspace-hypercyclicity is interesting,
including a nontrivial subspace-hypercyclic operator that is not
hypercyclic. There is a Kitai-like criterion that implies
subspace-hypercyclicity and although the spectrum of a
subspace-hypercyclic operator must intersect the unit circle, not every
component of the spectrum will do so. We show that, like
hypercyclicity, subspace-hypercyclicity is a strictly
infinite-dimensional phenomenon. Additionally, compact or hyponormal
operators can never be subspace-hypercyclic.
\end{abstract}

\maketitle 

\section{Introduction}

An operator on a Banach space is called {\em hypercyclic} if there is a
vector whose orbit under the operator is dense in the space; such a
vector is called a hypercyclic vector for the operator. It is somewhat
surprising that hypercyclic operators exist since they do not exist on
finite-dimensional Banach spaces. The study of hypercyclicity goes
back a long way, and has been studied in more general settings, for example in
topological vector spaces. A good reference are the survey papers of
K.-G.~Grosse-Erdmann~\cite{GE1,GE2}, which contain a guide to what is
known and not known about hypercyclicity (and universality, a more
general notion). In this paper, we study the problem of when the orbit
of a vector under an operator, intersected with a subspace, is dense
in that subspace.

We start by giving some context to our research. The first (and still
most famous) example of a hypercyclic operator on Banach spaces was
given by Rolewicz~\cite{rolewicz} in 1969: the example is the backward
shift on $\ell^p$ multiplied by a complex number of modulus bigger
than $1$. The first systematic study of hypercyclicity on Banach
spaces occurred in Kitai's doctoral dissertation ~\cite{kitai}, where
the famous Hypercyclicity Criterion was introduced. This criterion was
rediscovered later by Gethner and Shapiro in \cite{GeSh}. It  was an
open question for many years whether (a stronger version of) this
criterion was in fact equivalent to hypercyclicity: it was recently
shown by de la Rosa and Read \cite{dlRRe} that it is not. Furthermore,
Bayart and Matheron showed in \cite{BaMa} that the equivalence fails
on classical Banach spaces, and even on Hilbert space.

One reason the concept of hypercyclicity is interesting is because
it relates to the invariant subset problem: does every bounded
operator on a Banach space have a nontrivial invariant closed subset? An
operator has no nontrivial invariant closed subsets if and only if
every nonzero vector is hypercyclic. It is known that such operators
exist on Banach spaces \cite{read} but its existence is still an open
problem in Hilbert space. 

The question of hypercyclicity is really a dynamical one. As such, one
is interested in the possible behaviour of the orbit of a vector under
the operator. For example, if such an orbit is not dense, what other
``forms'' might it have? One possible direction for studying these
questions was undertaken in \cite{GE3} (as cited in \cite[Section
1a]{GE1}): namely, an operator is {\em hypercyclic for a nonempty
  closed set $A$} if there exists a vector $x$ such that the closure
of the orbit of $x$ under the operator contains $A$. 

Compare this concept with the remarkable result of Bourdon and
Feldman~\cite{BoFe}: if the orbit of a vector  under an operator is
somewhere dense, then it is everywhere dense; that is, if the closure
of the orbit has nonempty interior, then the operator must be
hypercyclic. Thus, if we want to study hypercyclicity for closed sets
$A$ we must restrict ourselves to cases where $A$ has empty interior;
for example, when $A$ is a (nontrivial) subspace.

In the present paper, we undertake the study of a special case. We ask
ourselves whether it is possible for a bounded operator on Hilbert
space to have the property that the orbit of some vector under the
operator ``touches a subspace enough times to fill it''. In more
technical jargon, if $\N_0:=\N \cup \{ 0 \}$, we ask whether the orbit
$$
\Orb(T,x):=\{ T^n x : n \in  \N_0\}
$$
has the property that $\Orb(T,x) \cap \calM$ is dense in $\calM$ for a
nonzero subspace $\calM$. We call this concept subspace-hypercyclicity.

We should note that,  although many of the results we present in this
paper are undoubtedly true for Banach spaces, we prefer to deal with the Hilbert 
space case exclusively, for the sake of simplicity. Also, in Hilbert space one 
might follow other avenues of research. For
example, given a subspace $\calM$ and the orthogonal projection $P$ on it,
one may ask if it is possible for $P(\Orb(T,x))$ to be dense in the
subspace $\calM$. We may investigate this question in further papers.

The paper is organized as follows. In the second section of this
paper, we will introduce formally the concept of
subspace-hypercyclicity and show some ``trivial'' examples. They will
be trivial in the sense that the subspace $\calM$ is invariant under
the operator.

Next, in the third section, we show the existence of nontrivial
examples. For this, we introduce the concept of subspace-transitivity
and we show that a ``Subspace-Hypercyclicity Criterion'' holds. We
also show, by giving a further example, that said criterion is not a
necessary condition. The examples we show in this section are all
based on the backward shift on $\ell^2$.

We prove in the fourth section of this paper that
subspace-hypercyclicity, like hypercyclicity, is a purely
infinite-dimensional concept. Namely, if an operator is subspace-hypercyclic 
for some subspace $\calM$, then $\calM$ is not finite-dimensional. Furthermore, 
there are no compact subspace-hypercyclic operators nor hyponormal
subspace-hypercyclic operators.

In the last section of this paper, we conclude with some open questions for future research.

\subsubsection*{Acknowledgement.} We thank the referee for remarks that improved our paper. In particular, we
thank the referee for the equivalence after Definition \ref{def:transitivity} that
allowed the proofs of Lemma~\ref{le:dense} and Theorem~\ref{th:kitai} to be greatly
simplified.

\section{Definition and Some Trivial Examples}

In this note $\Hil$ always denotes a separable Hilbert space over
$\C$, the field of complex numbers. Usually, it will be the case that
$\Hil$ is infinite-dimensional and we will explicitly indicate when a
result or definition only holds for finite or infinite
dimensions. According to usual practice, whenever we talk about a
subspace $\calM$ of $\Hil$ we will assume that $\calM$ is
topologically closed.

We will denote by $\BH$ the set of all bounded linear operators on
$\Hil$. We usually will refer to elements of $\BH$ as just
``operators''. The use of the symbol ``:='' indicates a definition.

We start with our main definition.

\begin{definition}
 Let $T \in \BH$ and let $\calM$ be a nonzero subspace of $\Hil$. We
 say that $T$ is {\em subspace-hypercyclic} for $\calM$ if there exists $x
 \in \Hil$ such that $\Orb(T,x) \cap \calM$ is dense in $\calM$. We
 call $x$ a {\em subspace-hypercyclic vector}.
\end{definition}

The definition above reduces to the classical definition of
hypercyclicity if $\calM = \Hil$. Observe also that the
subspace-hypercyclic vector $x$ is necessarily nonzero and we may assume,
if needed, that $x$ belongs to $\calM$.

We start by showing the simplest example of a subspace-hypercyclic
operator that is not hypercyclic.

\begin{example}\label{ex:sum}
Let $T$ be a hypercyclic operator on $\Hil$ with hypercyclic vector
$x$ and let $I$ be the identity operator on $\Hil$. Then, the operator
$T\oplus I: \Hil\oplus\Hil \to \Hil\oplus\Hil$ is subspace-hypercyclic
for the subspace $\calM:= \Hil \oplus \0$ with subspace-hypercyclic
vector $x \oplus 0$. Clearly, $T \oplus I$ is not hypercyclic.
\end{example}

The above example is trivial in the sense that $\calM$ is an
invariant subspace for $T\oplus I$. In fact, it is obvious that $T
\oplus I\restricted_{\calM}$ is in fact a hypercyclic operator.

The following example is trivial in the same sense.

\begin{example}\label{ex:eq}
Let $T$ be a hypercyclic operator on $\Hil$  with hypercyclic vector
$x$ and assume that $C \in \BH$ is nonzero and has closed range
$\calM$. If $A \in \BH$ satisfies the equation $AC=CT$, then it can easily be
checked that $A$ is subspace-hypercyclic for $\calM$ with
subspace-hypercyclic vector $Cx$. 
\end{example}

The subspace-hypercyclicity of the above example is obviously due to
the fact that $A\restricted_\calM$ is a hypercyclic operator, which
can be proven easily. (Alternatively, this follows from well-known facts for
transitive and hypercyclic operators: two references are the paper \cite[Lemma
2.1]{MP} and the notes \cite[Proposition 1.13]{shapiro}).

From the above two examples, one can obtain subspace-hypercyclic
operators from known hypercyclic operators, the simplest examples of
which are multiples of the backward shift. 

Recall that on $\ell^2$, the Hilbert space of all square summable
complex sequences, the backward shift $B$ is defined as
$$
B(x_0, x_1, x_2, x_3,\ldots):=(x_1, x_2, x_3, x_4, \ldots).
$$
As we mentioned in the introduction, it was shown in \cite{rolewicz}
that $\lambda B$ is a hypercyclic operator if $|\lambda|>1$. Thus
setting $T=\lambda B$ in Example \ref{ex:sum} provides a concrete
example of a subspace-hypercyclic operator.

One can find another concrete example by setting $T=\lambda B$ and
finding operators $A$ and $C$ that satisfy the conditions of Example
\ref{ex:eq}, but one should note that if $T$ is a multiple of the
backward shift and $A$ is injective, $C$ would have to be zero if it
is to satisfy the equation $A C = C T$. (The proof is easy, see
\cite{MA} if needed.)

Nevertheless, one can obtain an interesting example. For the
following, the reader should be familiar with some basic facts about
the Hardy-Hilbert space $\mathbf{H}^2$. Let $\D$ be the open unit disk
in $\C$ and $\phi: \D \to \D$ be an analytic funtion. Let $T_\phi$ denote
the analytic Toeplitz operator on $\mathbf{H}^2$ defined by $T_\phi (f) = \phi f$,
let $B$ be the backward shift on $\mathbf{H}^2$ and let $C_\phi$ denote the composition
operator on $\mathbf{H}^2$ defined by $C_\phi (f) = f \circ \phi$
(see, for example, \cite{MR} for the definition of $\mathbf{H}^2$ and
the basic properties of the above operators).

\begin{example}\label{ex:toep}
Let $\phi \in \mathbf{H}^2$ be an inner function with
$\phi(0)=0$ and $\phi$ not the identity function. Then $T_\phi^*
C_\phi = C_\phi B$ and hence $\lambda T_\phi^*$ is
subspace-hypercyclic for the subspace $\ran C_\phi$ if $|\lambda|> 1$.
\end{example}

On the other hand, observe that $\lambda T_\phi^*$ is hypercyclic, by the
characterization given by Godefroy and Shapiro~\cite{GoSh}. Note that
a subspace-hypercyclic vector for $\lambda T_\phi^*$ with respect to
$\ran C_\phi$ is necesarily not a hypercyclic vector for $\lambda
T_\phi^*$. 

As in the case of hypercyclicity, analytic Toeplitz operators can
never be subspace-hypercyclic. Indeed, suppose $T_\phi$ is an analytic
Toeplitz operator which is subspace-hypercyclic for a subspace
$\calM$. Let $k_\lambda \in {\mathbf H}^2$ be the reproducing kernel
for $\lambda \in \D$. Then, as is well-known, $k_\lambda \in
\ker(T_\phi^* - \conj{\phi(\lambda)})$. By Proposition
\ref{prop:kernel} below, we have that $k_\lambda \subseteq
\calM^\perp$ for all $\lambda \in \D$. The fact that the reproducing
kernels have a dense span in $\mathbf{H}^2$ implies that $\mathbf{H}^2
\subseteq \calM^\perp$, and hence that $\calM=\0$.

\section{Some nontrivial examples and a subspace-hypercyclicity
  criterion}

We would like to find examples of subspace-hypercyclic operators for a
subspace $\calM$ such that $\calM$ is not invariant under the
operator. The results presented in this section, including a
subspace-hypercyclicity criterion, will help us achieve that goal.

Let us denote the set of subspace-hypercyclic vectors for $\calM$ by
$$
\HC(T,\calM):=\{ x \in \Hil \, : \, \Orb(T,x) \cap \calM \hbox{ is
  dense in } \calM \}.
$$

The following lemma holds true in the classical setting and the proof
in our setting is the same. We include here its short proof for the sake of completeness.

\begin{lemma}
Let $T \in \BH$ and let $\calM$ be a nonzero subspace of $\Hil$. Then
$$
\HC(T,\calM) = \bigcap_{j=1}^{\infty} \bigcup_{n=0}^{\infty} T^{-n}(B_j)
$$
where $\{ B_j \}$ is a countable open basis for the relative topology
of $\calM$ as a subspace of $\Hil$.
\end{lemma}
\begin{proof}
We have that $x \in \bigcap_{j=1}^{\infty}
\bigcup_{n=0}^{\infty} T^{-n}(B_j)$ if and only if, for any 
$j\in \N$, there exists a number $n \in \N_0$ such that $T^nx \in B_j$. But since
$\{B_j\}$ is a basis for the relative topology of $\calM$, this occurs if
and only if $\Orb(T,x)\cap \calM$ is dense in $\calM$ or,
equivalently, if $x \in HC(T,\calM)$.
\end{proof}

Hence, if the set in the display above is nonempty, $T$ is
subspace-hypercyclic for $\calM$. Our following lemma will obtain much
more than what is needed  to imply the nonemptiness of said set. The following
definition will be convenient.

\begin{definition}\label{def:transitivity}
Let $T \in \BH$ and let $\calM$ be a nonzero subspace of $\Hil$. We say that
$T$ is {\em subspace-transitive with respect to $\calM$} if for all
nonempty sets $U \subseteq \calM$ and $V\subseteq \calM$, both relatively
open, there exists $n \in \N_0$ such that $T^{-n}(U) \cap V$ contains a
relatively open nonempty subset of $\calM$.
\end{definition}

\subsection*{Note added:} The referee has kindly pointed out the
following equivalences for our definition above, which greatly
simplify Lemma~\ref{le:dense} and Theorem~\ref{th:kitai} below.

\begin{theorem}\label{th:equiv}
Let $T\in \BH$ and let $\calM$ be a nonzero subspace of $\Hil$. Then
the following conditions are equivalent:
\begin{enumerate}
\item The operator $T$ is subspace-transitive with respect to $\calM$.
\item For any nonempty sets $U \subseteq \calM$ and $V\subseteq
  \calM$, both relatively open, there exists $n \in \N_0$ such that
  $T^{-n}(U) \cap V$ is a relatively open nonempty subset of $\calM$. 
\item  For any nonempty sets $U \subseteq \calM$ and $V\subseteq
  \calM$, both relatively open, there exists $n \in \N_0$ such that
  $T^{-n}(U) \cap V$ is nonempty and $T^n(\calM) \subseteq \calM$. 
\end{enumerate}
\end{theorem}
\begin{proof}
The implication (ii) $\implies$ (i) is obvious. The impĺication (iii)
$\implies$ (ii) is obvious once one observes that the operator
$T^n \restricted_{\calM} : \calM \to \calM$ is continuous and hence
$T^{-n}(U)$ is relatively open in $\calM$ if $U$ is relatively open in
$\calM$.

We will show that (i) $\implies$ (iii). Let $T$ be subspace-transitive
with respect to $\calM$ and let $U$ and $V$ be nonempty relatively
open subsets of $\calM$. By Definition \ref{def:transitivity} above,
it follows that there exists $n\in \N_0$ such that  $T^{-n}(U) \cap V$
contains a relatively open nonempty set, say $W$. Thus, in particular,
$T^{-n}(U) \cap V$ is nonempty. Now, let $x \in \calM$. Since $W
\subseteq T^{-n}(U)$, it follows that $T^n(W) \subseteq \calM$. Take
$x_0 \in W$. Since $W$ is relatively open and $x \in \calM$, for $r>0$
small enough, we have $x_0+r x \in W$, and hence $T^n(x_0+r x) =
T^n(x_0) + r T^n(x) \in \calM$. Since $T^n(x_0) \in \calM$,
subtracting it and dividing by $r$ leads to $T^n(x)\in \calM$,
showing $T^n(\calM) \subseteq \calM$ and finishing the proof.
\end{proof}

The following lemma will achieve ``half'' of the classical equivalence
of topological transitivity and hypercyclicity.

\begin{lemma}\label{le:dense}
Let $T \in \BH$ and let $\calM$ be a nonzero subspace of $\Hil$. Assume that
$T$ is subspace-transitive with respect to $\calM$. Then, 
$$
\bigcap_{j=1}^{\infty} \bigcup_{n=0}^{\infty} T^{-n}(B_j) \cap
\calM,
$$
is a dense subset of $\calM$. Here $\{ B_j \}$ is a countable open
basis for the (relative) topology of $\calM$.
\end{lemma}
\begin{proof}

By Theorem~\ref{th:equiv} above, for each $j$ and $k$, there exists
$n_{j,k} \in \N_0$ such that the set $T^{-n_{j,k}}(B_{j}) \cap B_{k}$
is nonempty and relatively open. Hence, the set
$$
A_j:=\bigcup_{k=1}^\infty T^{-n_{j,k}}(B_{j}) \cap B_{k}
$$
is relatively open. Furthermore, each $A_j$ is dense, since it
intersects each $B_k$. By the Baire Category Theorem, this implies
that
$$
\bigcap_{j=1}^\infty A_j=\bigcap_{j=1}^\infty \bigcup_{k=1}^\infty T^{-n_{j,k}}(B_{j}) \cap B_{k} 
$$
is a dense set. But clearly,
$$
\bigcap_{j=1}^\infty \bigcup_{k=1}^\infty T^{-n_{j,k}}(B_{j}) \cap
B_{k} \subseteq \bigcap_{j=1}^\infty \bigcup_{n=1}^\infty
T^{-n}(B_{j}) \cap \calM,
$$
and the result follows.
\end{proof}

The above lemmas clearly combine to imply the following theorem.

\begin{theorem}\label{th:tran_hyper}
Let $T \in \BH$ and let $\calM$ be a nonzero subspace of $\Hil$. If
$T$ is subspace-transitive for $\calM$ then $T$ is
subspace-hypercyclic for $\calM$.
\end{theorem}

We will show (see the comment after Example~\ref{ex:halmos} below) that the converse of the
above theorem is not true.

The following theorem is a subspace-hypercyclicity criterion, stated
in the style of \cite{GE2}

\begin{theorem}\label{th:kitai}
Let $T\in \BH$ and let $\calM$ be a nonzero subspace of $\Hil$. Assume
there exist $X$ and $Y$, dense subsets of $\calM$, and an increasing
sequence of positive integers $\{n _k \}$ such that
\begin{enumerate}
\item $T^{n_k} x \to 0$ for all $x \in X$, 
\item for each $y \in Y$, there exist a sequence $\{ x_k \}$ in $\calM$
  such that
$$
x_k \to 0 \quad \hbox{ and } \quad T^{n_k} x_k \to y,
$$
\item $\calM$ is an invariant subspace for $T^{n_k}$ for all $k \in \N$.
\end{enumerate}
Then $T$ is subspace-transitive with respect to $\calM$ and hence $T$
is subspace-hypercyclic for $\calM$.
\end{theorem}
\begin{proof}
Let $U$ and $V$ be nonempty relatively open subsets of $\calM$. We
will show that there exists $k \in \N_0$ such that $T^{-n_k}(U) \cap V$
is nonempty. By Theorem~\ref{th:equiv}, since $T^{n_k}(\calM)
\subseteq \calM$, it will follow that $T$ is subspace-transitive with
respect to $\calM$.

Since $X$ and $Y$ are dense in $\calM$, there exists $v \in X \cap V$
and $u \in Y \cap U$. Furthermore, since $U$ and $V$ are relatively
open, there exists $\varepsilon >0$ such that the $\calM$-ball centered at $v$
of radius $\varepsilon$ is contained in $V$ and the  $\calM$-ball
centered at $u$ of radius $\varepsilon$ is contained in $U$.

By hypothesis, given these $v \in X$ and $u \in Y$, one can choose $k$ large
enough such that there exists $x_k \in \calM$ with
$$
\| T^{n_k} v \| < \frac{\varepsilon}{2}, \quad \| x_k  \| <
\varepsilon, \quad \hbox{ and } \quad \| T^{n_k} x_k - u  \| <
\frac{\varepsilon}{2}.
$$

We have:
\begin{itemize}

\item $v + x_k \in V$. Indeed, since $v \in \calM$ and $x_k
\in \calM$, it follows that $v+x_k \in \calM$. Also, since 
$$
\| (v + x_k) - v \| = \| x_k \| < \varepsilon,
$$
it follows that $v + x_k$ is in the $\calM$-ball centered at $v$ of
radius $\varepsilon$ and hence $v + x_k \in V$.

\item $T^{n_k}( v + x_k) \in U$. Indeed, since $v$ and $x_k$ are  in $\calM$
  and $\calM$ is invariant under $T^{n_k}$, it follows that $T^{n_k}( v + x_k) \in
\calM$. Also,
$$
\|T^{n_k} (v + x_k) - u \| \leq \| T^{n_k} v \| + \| T^{n_k} x_k -u \|
< \frac{\varepsilon}{2} + \frac{\varepsilon}{2} = \varepsilon,
$$
and hence $T^{n_k} (v + x_k)$ is in the $\calM$-ball centered at $u$ of
radius $\varepsilon$ and thus $T^{n_k} (v + x_k) \in U$.
\end{itemize}
The two facts above imply that $v + x_k \in T^{-n_k} (U) \cap V$ and
hence this set is nonempty. 
\end{proof}

A natural question to ask is whether condition (iii) in the theorem above is really necessary. After a preliminary version of our paper was distributed, Le \cite{le} proved an alternative Subspace-Hypercyclicity Criterion with a condition weaker than condition (iii) above. Evenmore, Le shows that Theorem~\ref{th:kitai} is false if only conditions (i) and (ii) hold.

The referee has pointed out to us that, with the conditions of Theorem~\ref{th:kitai}, the
sequence of operators $T^{n_k}\!\!\restricted_\calM : \calM \to \calM$ is
universal, as follows from (for example) \cite[Theorem 2]{GE1}, and thus
$T$ is subspace-hypercyclic for $\calM$. Note that Theorem
\ref{th:kitai} is stronger because it shows that $T$ is, in fact,
subspace-transitive for $\calM$.

In general, let $T:\Hil \to \Hil$ be an operator for which there exists an
increasing sequence $\{n_k\}$ of natural numbers such that
$T^{n_k}(\calM) \subseteq \calM$. If the sequence
$T^{n_k}\!\!\restricted_ \calM : \calM \to \calM$ is universal, it
follows that $T$ is subspace-hypercyclic for $\calM$. Contrast this
with Example \ref{ex:halmos} which will show that an operator $T$ can
be subspace-hypercyclic for a subspace $\calM$ not invariant for any
power of the operator.

The following is our first example of a subspace-hypercyclic operator
for a subspace $\calM$ such that $\calM$ is not invariant for the
operator. Recall that the {\em forward shift} $S$ on $\ell^2$  is the
operator defined by
$$
S(x_0,x_1, x_2, x_3, \ldots):= (0, x_0, x_1, x_2, \ldots).
$$
Clearly $B S$ equals the identity on $\ell^2$. Observe also that $S$
is an isometry.

\begin{example}\label{ex:kitai}
Let $\lambda \in \C$ be of modulus greater than $1$ and consider
$T:=\lambda B$ where $B$ is the backward shift on $\ell^2$. Let
$\calM$ be the subspace of $\ell^2$ consisting of all sequences with
zeroes on the even entries; that is, 
$$
\calM:=\left\{  \{ a_n \}_{n=0}^{\infty} \in \ell^2 \, : \, a_{2k}=0 \hbox{ for all }
k \right\}.
$$
Then $T$ is subspace-hypercylic for $\calM$.
\end{example}
\begin{proof}
One can see that $T^2$ on $\calM$ behaves like the hypercyclic
operator $\lambda^2 B$ on $\ell^2$ and hence $T$ is
subspace-hypercyclic for $\calM$.

We will apply Theorem \ref{th:kitai} to give an alternative proof. Let
$X=Y$ be the subset of $\calM$ consisting of all finite sequences;
i.e., those sequences that only have a finite number of nonzero
entries: this clearly is a dense subset of $\calM$. Let $n_k:= 2
k$. Let us check that conditions (i), (ii) and (iii) hold.

Let $x \in X$. Since $x$ only has finitely many nonzero
entries, $T^{2k}x$ will be zero eventually for $k$ large enough. Thus (i) holds.

Let $y \in Y$ and define $x_k:= \frac{1}{\lambda^{2k}} S^{2k}
y$, where $S$ is the forward shift on $\ell^2$. Each $x_k$ is
in $\calM$ since the even entries of $y$ are shifted by $S^{2k}$ into the even
entries of $x_k$.  We have
$$
\| x_k \|= \frac{1}{|\lambda|^{2k}} \| y \|,
$$
and thus it follows that $x_k \to 0$, since $|\lambda|>1$. Also, because
$$
T^{2k} x_k = (\lambda B)^{2k} x_k =  (\lambda B)^{2k}
\frac{1}{\lambda^{2k}} S^{2k} y = y,
$$
we have that condition (ii) holds. 

That condition (iii) holds follows from the fact that if a vector has
a zero entry on all even positions then it will also have a zero
entry on all even positions after the application of the backward
shift any even number of times. 

The subspace-hypercyclicity of $T$ now follows.
\end{proof}

As we commented before, the above is our first example of a
subspace-hypercyclic operator $T$ for which $\calM$ is not invariant under
$T$. Observe that, nevertheless, the subspace $\calM$ in the above example
is invariant for $T^2$.

The example above can be generalized. Let $a\in \N_0$ and $b\in \N$ be
some fixed numbers with $a < b$ and consider the subspace $\calM$ of $\ell^2$
consisting of all sequences with zeroes on the entries indexed by the
set $\{ a + k b \, : \, k \in \N_0\}$. Then, the argument above holds
for $X$ and $Y$ both equal to the set of all sequences in $\calM$ with
finitely many nonzero entries and $n_k:=b k$. Hence, for $|\lambda|>1$, the 
operator $\lambda B$ is subspace-hypercyclic for $\calM$. In this
case, the space $\calM$ is $T^b$ invariant but not $T$ invariant. 

As the referee so clearly expressed, all examples given so far are in
some sense trivial, ``because they contain some more or less hidden element
of hypercyclicity''. In Examples \ref{ex:sum} and \ref{ex:eq} the
subspace is invariant under the operator, and hence the operator
restricted to the subspace is hypercyclic. In Example \ref{ex:kitai}
the subspace is invariant under a power of the operator. In Theorem
\ref{th:kitai}, the subspace is invariant for the operators $T^{n_k}$.

The next example, although based on a hypercyclic operator, does not
fit into any of the categories above, since the subspace is not
invariant for any power of the operator.

\begin{example}\label{ex:halmos}
Let $\lambda \in \C$ be of modulus greater than $1$ and let $B$ be
the backward shift on $\ell^2$. Let $m \in \N$ and $\calM$ be the
subspace of $\ell^2$ consisting of all sequences with zeroes on the
first $m$ entries; that is, 
$$
\calM:=\left\{  \{ a_n \}_{n=0}^\infty \in \ell^2 \, : \, a_{n}=0
  \hbox{ for } n < m \right\}.
$$
Then $\lambda B$ is subspace-hypercylic for $\calM$.
\end{example}
\begin{proof}
The argument used here is really the same as the one Rolewicz \cite{rolewicz}
originally used to show hypercyclicity of $\lambda B$. We base our
proof on the expositions of Halmos \cite[p.~286]{Halmos} and of
Jim\'enez-Mungu\'ia \cite{JM}.

First of all, for a complex sequence $h=\{h_j\}_{j=0}^\infty$ with finitely-many
nonzero entries, we define its {\em length} as
$$
|h|:=\min\{ s \in \N_0 \, : \, h_k = 0 \hbox{ for all } k \geq s \}.
$$

We can choose a countable dense subset of $\calM$, called $\{ f_j \}$,
consisting of sequences which have at most a finite number of nonzero entries. 

Define a (necessarily increasing) sequence of integers $k_j$ inductively as
follows. Let $k_0=0$ and for each $j \in \N$ choose $k_j$ in such a way that 
\begin{equation}\label{eq:cond}
\frac{\| f_j  \|}{|\lambda|^{k_j-k_{j-1}}} \leq \frac{1}{|\lambda|^j},
\end{equation}
and such that $k_j > k_{j-1} + |f_{j-1}|$.

Define the vector $f$ by
$$
f:=\sum_{j=0}^\infty \frac{S^{k_j} f_j}{\lambda^{k_j}},
$$
where $S$ is the forward shift.

We must first show that $f \in \ell^2$. It follows from inequality
\eqref{eq:cond} that
$$
\left\|  \frac{S^{k_j} f_j}{\lambda^{k_j}} \right\| 
= \frac{\|f_j\|}{|\lambda|^{k_j}} 
\leq \frac{1}{|\lambda|^{j + k_{j-1}}} 
\leq \frac{1}{|\lambda|^j},
$$
and hence the infinite sum converges in norm to an $\ell^2$ vector.

Let $n \in \N$. Since for all $j$ the condition $k_j > k_{j-1} +
|f_{j-1}|$ holds, it follows that $k_n > k_j + | f_j|$ for all $j < n$
and hence
$$
(\lambda B)^{k_n} \frac{S^{k_j} f_j}{\lambda^{k_j}} = 0.
$$
This implies that $(\lambda B)^{k_n}f$ ``starts'' with the vector
$f_n$ and hence that $(\lambda B)^{k_n}f$ is in $\calM$. 

The condition  $k_j > k_{j-1} + |f_{j-1}|$ also implies that the norm
of the difference $(\lambda B)^{k_n}f - f_n$ is given by
\begin{eqnarray*}
\|  (\lambda B)^{k_n}f - f_n  \|^2 
&=& \sum_{j=n+1}^\infty \left\| \frac{S^{k_j-k_n} f_j }{\lambda^{k_j-k_n}} \right\|^2 \\
&=& \sum_{j=n+1}^\infty \left(\frac{\|f_j\| }{|\lambda|^{k_j-k_n}} \right)^2 \\
&\leq& \sum_{j=n+1}^\infty \left(\frac{\|f_j\| }{|\lambda|^{k_j-k_{j-1}}} \right)^2 \\
&\leq& \sum_{j=n+1}^\infty \frac{1}{|\lambda|^{2j}}.
\end{eqnarray*}

Let $h \in \calM$. Given $\epsilon >0$, choose $N$ such that $\| h -
f_N \| < \frac{\epsilon}{2}$ and such that 
$$
\left( \sum_{j=N+1}^\infty \frac{1}{|\lambda|^{2j}} \right)^{1/2} < \frac{\epsilon}{2}.
$$
It then follows that 
$$
\|(\lambda B)^{k_N}f - h \| < \epsilon,
$$
and hence that $\Orb(\lambda B, f) \cap \calM$ is dense in $\calM$.
\end{proof}

In the example above, it is clear that it is impossible to find an
increasing sequence of integers $n_k$ such that $\calM$ is invariant
for $T^{n_k}$ (since clearly, $\calM$ is not invariant for $T^n$ for
any $n$). Thus, condition (iii) in Theorem \ref{th:kitai} is
not necessary. 

Observe that the operator above does not satisfy
Theorem~\ref{th:equiv} and hence this implies that
subspace-hypercyclicity for a subspace $\calM$  does not imply
subspace-transitivity for $\calM$. After a preliminary version of this paper was distributed, another example of a 
subspace-hypercyclic operator for a subspace $\calM$ which is not 
subspace-transitive for $\calM$ was given in \cite{le}.

It should be noted that the procedure used in Example \ref{ex:halmos}
above, could have been used to find a subspace-hypercyclic vector for
the operator in Example \ref{ex:kitai}. 

Let us contrast the behaviour of the subspace-hypercyclic vector in
the above two examples. In Example \ref{ex:kitai}, the orbit of any
subspace-hypercyclic vector $x$ under $T$ goes in and out of the space
$\calM$ at regular intervals. In Example \ref{ex:halmos}, by
judiciously choosing the sequence of finite vectors $\{ f_j \}$ and
the sequence of natural numbers $\{ k_j \}$, the orbit of $f$ under
$T$ goes in and out of the space $\calM$ at irregular intervals;
namely, one can find arbitrarily long consecutive elements of the
orbit that stay inside the space and arbitrarily long consecutive
elements of the orbit that stay outside the space.

We do not know if the vector $f$ in Example~\ref{ex:halmos} is
hypercyclic. An easy way to obtain a subspace-hypercyclic
operator which is not hypercyclic and has the properties of
Example~\ref{ex:halmos} above, follows.

\begin{example}
Let $|\lambda| > 1$, and consider the operator $T:=(\lambda B) \oplus
I$ on $\Hil:=\ell^2 \oplus \ell^2$. Let $\calM$ be as in Example
\ref{ex:halmos} and let $f$ be a subspace-hypercyclic vector for
$\calM$. Define $\calN:=\calM \oplus \{ 0 \}$. Then $f
\oplus 0$ is a subspace-hypercyclic vector for $\calN$, but $f \oplus
0$ is not hypercyclic for $\Hil$. Also, $\calN$ is not an invariant
subspace for $T^k$ for any $k$.
\end{example}

\section{Finite Dimensions}

The following easy observation will be useful.

\begin{proposition}\label{prop:restriction}
Let $T\in \BH$ be subspace-hypercyclic for $\calM$. If $\calN$ is an
invariant subspace for $T$ and $\calM \subseteq \calN$, then
$T\restricted_{\calN} : \calN \to \calN$ is subspace-hypercyclic for $\calM$.
\end{proposition}

We remind the reader of the following well known definitions.

\begin{definition}
Let $\calM$ and $\calN$ be subspaces of $\Hil$. If $\calM \cap \calN =
\{0 \}$ and $\calM + \calN = \Hil$ we say that $\calM$ and $\calN$ are
{\em complementary}. 
\end{definition}

\begin{definition}
Let $\calM$ and $\calN$ be complementary subspaces of $\Hil$. The {\em
  projection onto $\calM$ along $\calN$} is the function $P:\Hil \to
\Hil$ defined as
$$
P(x+y) = x,
$$
where $x\in \calM$ and $y \in \calN$.
\end{definition}

It can easily be shown that $P$ is a bounded operator. This allows us
to obtain the following theorem. An analogous result for hypercyclic
operators is due to Herrero~\cite{herrero}.

\begin{theorem}\label{theo:complementary}
Let $\calM$ and $\calN$ be complementary subspaces of $\Hil$ and let
$P$ be the projection onto $\calM$ along $\calN$. Let $T \in \BH$ and
suppose that $\calN$ is invariant under $T$. If $T$ is subspace-hypercyclic
for some $\calL \subseteq \calM$, then $P T\restricted_{\calM}$ is
subspace-hypercyclic for $\calL$.
\end{theorem}
\begin{proof}
Assume $T$ is subspace-hypercyclic for $\calL$ with
subspace-hypercyclic vector $x \in \calL$. Since $\Orb(T,x) \cap
\calL$ is dense in $\calL$, and $\calL \subseteq \calM$ it follows
that $P(\Orb(T,x)) \cap \calL$ is dense in $\calL$.

Also, since $\calN$ is an invariant subspace for $T$, it follows that
$PTP=PT$ and hence that $(PT)^k = P T^k$ for all $k \in \N$. Thus
$$
P ( \Orb(T,x) ) = \Orb(P T\restricted_{\calM},x).
$$
It follows that $\Orb(P T\restricted_{\calM},x) \cap  \calL$ is dense
in $\calL$, as desired.
\end{proof}

The following proposition was shown to us by A.~Peris.

\begin{theorem}\label{th:hyp_spec}
Let $T\in \BH$. If $T$ is subspace-hypercyclic for some subspace then
$\sigma(T) \cap S^1 \neq \emptyset$.
\end{theorem}
\begin{proof}
Assume the intersection is empty. Then, there exist (possibly empty) sets $K_1$ and $K_2$
such that $\sigma(T)=K_1 \cup K_2$ with $K_1 \subseteq \D$ and $K_2
\subseteq \conj{\D}^c$. By the Riesz Decomposition Theorem
\cite[Theorem 2.10]{RR}, there exist complementary invariant subspaces
$\calM_1$ and $\calM_2$ such that
$$
\sigma\left(T\restricted_{\calM_1}\right) \subseteq K_1 \quad \hbox{
  and } \quad \sigma\left(T\restricted_{\calM_2}\right) \subseteq K_2.
$$
Let $x \in \Hil$. Then, there exist $x_1 \in \calM_1$ and $x_2 \in
\calM_2$ such that $x=x_1+x_2$. If $x_2$ was equal to zero, then
$$
T^n x = T^n x_1= \left( T\restricted_{\calM_1} \right)^n x_1
$$
which converges to zero because
$\sigma\left(T\restricted_{\calM_1}\right) \subseteq \D$. Thus,
$\Orb(T,x)$ is bounded and hence its intersection with any subspace
cannot be dense in that subspace.

Assume $x_2$ is not equal to zero. We have
$$
\| T^n x \| = \| T^n x_1 + T^n x_2 \| \geq \| T^n x_2 \| - \| T^n x_1 \|,
$$
and as before, $\| T^n x_1 \|$ goes to zero. Since
$\sigma\left(T\restricted_{\calM_2}\right) \subseteq \conj{\D}^c$, it
follows that $\| T^n x_2 \|$ goes to infinity and hence $\| T^n x \|$
goes to infinity. Thus only finitely many elements of $\Orb(T,x)$
intersect any bounded set, and hence $\Orb(T,x)$ intersected with a
subspace cannot be dense in that subspace.
\end{proof}

As we did in Example \ref{ex:sum}, it can easily be shown that the
operator $(2 B) \oplus (3I): \ell^2 \oplus \ell^2 \to \ell^2 \oplus
\ell^2$ is subspace hypercyclic for $\calM:=\ell^2 \oplus \0$ with
subspace hypercyclic vector $f \oplus 0$, where $f$ is a hypercyclic
vector for $2B$. Observe that $\sigma((2 B) \oplus (3I))$ is the
closed disk of radius $2$ union the singleton $\{3\}$. Thus, it is not
true that every component of the spectrum must intersect the unit
circle for a subspace-hypercyclic operator, as is the case for
classical hypercyclicity.

Is there any spectral restriction besides the one given by Theorem
\ref{th:hyp_spec}? Namely, given an arbitrary nonempty compact subset
$K$ of $\C$ which intersects $S^1$, is there a subspace-hypercyclic
operator with that set as its spectrum?

Yes. Let $K_0$ be a component of $K$ which intersects
$S^1$. By a theorem of Shkarin~\cite{shkarin}, one can construct a
hypercyclic operator $T$ with spectrum $K_0$. Take the direct sum of
$T$ with an operator whose spectrum is the closure of $K\setminus K_0$
to obtain the desired operator. This proof was communicated to us
independently by A.~Peris and by the referee.

\begin{proposition}\label{prop:kernel}
Let $T \in \BH$ be subspace-hypercyclic for $\calM$. Then $\ker(T^* -
\lambda) \subseteq \calM^\perp$ for all $\lambda \in \C$.
\end{proposition}
\begin{proof}
Assume that $\Orb(T,x) \cap \calM$ is dense in $\calM$. Fix $\lambda
\in \C$ and let $y$ be in $\ker (T^*-\lambda)$.  Let $\varphi
: \calM \to \C$ be the functional defined by $\varphi(x) := \left<
  x, y \right>$. Clearly $\varphi$ is surjective if and only if $y
\notin \calM^\perp$.

Observe that
$$
\left< T^n x, y \right> = \left< x, T^{*n} y \right> =  \left< x,
  \lambda^n y \right> = \conj{\lambda}^n \left<x, y \right>,
$$
and hence
\begin{equation}\label{eq:phi_orb}
\varphi(\Orb(T,x) \cap \calM) = \left\{ \conj{\lambda}^n \left< x, y
  \right> \, : \, \hbox{ there exists $n$ such that } T^n x \in  \calM
\right\}.
\end{equation}
But, if $\Orb(T,x) \cap \calM$ is dense in $\calM$, then
$\varphi(\Orb(T,x) \cap \calM)$ must be dense in $\C$, unless
$\varphi$ is not surjective. Since the set in equation
\eqref{eq:phi_orb} is clearly not dense in $\C$, it follows that
$\varphi$ is not surjective and hence $y \in \calM^\perp$.
\end{proof}

The result above can be generalized. We will need the following fact.

\begin{lemma}\label{lemma:formula}
Let $T \in \BH$ and $\lambda \in \C$. If $(T-\lambda)^p y = 0$ for
some $p \in \N$, then for $n \geq p$
$$
T^n y = \sum_{k=1}^p \binom{p}{k} \binom{n}{p} \frac{k}{n-p+k}
(-1)^{k-1} \lambda^{n-p+k} \,  T^{p-k} y.
$$
\end{lemma}
\begin{proof}
This follows by a straightforward (but tedious) induction argument.
\end{proof}

The case $p=1$ of the following proposition was shown in Proposition \ref{prop:kernel} above.

\begin{proposition}\label{prop:gen_kernel}
Let $T \in \BH$ be subspace-hypercyclic for $\calM$. Then $\ker(T^* -
\lambda)^p \subseteq \calM^\perp$ for all $\lambda \in \C$ and all $p
\in \N$.
\end{proposition}
\begin{proof}
Assume that $\Orb(T,x) \cap \calM$ is dense in $\calM$. Set $\lambda
\in \C$, $p \in \N$ and let $y$ be in $\ker (T^*-\lambda)^p$.
As before, let $\varphi : \calM \to \C$ be the functional defined by
$\varphi(x) := \left< x, y \right>$. Again,  $\varphi$ is surjective
if and only if $y \notin \calM^\perp$.

Observe that, for $n \geq p$ we have, by Lemma \ref{lemma:formula}, 
\begin{eqnarray*}
\left< T^n x, y \right> 
&=& \left< x, T^{*n} y \right> \\
&=&  \sum_{k=1}^p \binom{p}{k} \binom{n}{p} \frac{k}{n-p+k}
(-1)^{k-1} \conj{\lambda}^{n-p+k} \left< x,  T^{*(p-k)} y \right>\\
&=& \conj{\lambda}^{n-p} \sum_{k=1}^p \binom{p}{k} \binom{n}{p} \frac{k}{n-p+k}
(-1)^{k-1} \conj{\lambda}^{k} \left<T^{p-k} x, y \right>.
\end{eqnarray*}
Since 
$$
 \binom{n}{p} = \frac{n(n-1) \cdots (n-p+k) \cdots (n-p+1)}{p!},
$$
for each $k$ less than or equal to $p$, it follows that $\ds\binom{n}{p}
\frac{1}{n-p+k}$ is a polynomial in the variable $n$ of degree $p-1$. Hence
$$
\left< T^n x, y \right> = \conj{\lambda}^{n-p} Q(n) 
$$
where $Q(n)$ is a polynomial in the variable $n$ of degree at most
$p-1$ with complex coefficients (the coefficients depend, of course,
on $\lambda$, $p$, $T$, $x$ and $y$).

Hence we have
\begin{multline*}
\varphi(\Orb(T,x) \cap \calM) = \\
\left\{ \left< T^j x, y \right> \, : \, \hbox{ there exists $j=0, 1, \dots,
  p-1$ such that } T^j x \in \calM \right\}  \\
\cup \left\{ \conj{\lambda}^{n-p}  Q(n) \, : \, \hbox{ there exists $n \geq p$
  such that } T^n x \in  \calM \right\}.
\end{multline*}
It can be checked that the above set is never dense in $\C$.

But, if $\Orb(T,x) \cap \calM$ were dense in $\calM$, then
$\varphi(\Orb(T,x) \cap \calM)$ would be dense in $\C$, which we just
agreed was impossible. Thus  $\varphi$ is not surjective and
hence $y \in \calM^\perp$. This finishes the proof.
\end{proof}

\begin{theorem}\label{theo:finite_dim}
Let $\Hil$ be finite-dimensional. If $T\in \BH$ then $T$ is not
subspace hypercyclic for any $\calM$.
\end{theorem}
\begin{proof}
Since $\Hil$ is finite-dimensional and $T^* \in \BH$, it is well known (e.g.,
\cite[p.~174]{Axler}) that there exist complex numbers $\lambda_1, \lambda_2,
\dots, \lambda_s$ and natural numbers $p_1, p_2, \dots, p_s$ such that
$\Hil$ is the direct sum of the subspaces
$$
\ker(T^*-\lambda_1)^{p_1}, \ker(T^*-\lambda_2)^{p_2}, \dots, \ker(T^*-\lambda_s)^{p_s}.
$$
If $T$ was subspace-hypercyclic for some $\calM \neq \{0 \}$, then by
Proposition \ref{prop:gen_kernel}, for each
$j$, we have
$$
\ker(T^* - \lambda_j)^{p_j} \subseteq \calM^\perp.
$$
Hence $\Hil \subseteq \calM^\perp$ and thus $\calM=\{0\}$, a contradiction.
\end{proof}

\begin{theorem}
Let $T \in \BH$. If $T$ is subspace-hypercyclic for $\calM$, then
$\calM$ is not finite-dimensional.
\end{theorem}
\begin{proof}
Assume $T$ was subspace-hypercyclic for a finite-dimensional subspace
$\calM$ and let $x\in \calM$ such that $\Orb(T,x) \cap \calM$ is dense
in $\calM$. It then follows that the infinite set $\Orb(T,x) \cap
\calM$ in the finite-dimensional space $\calM$ has a finite
subset of (nonzero) linearly dependent vectors, say $\{ T^{n_1}x,
T^{n_2}x, \dots, T^{n_k}x \}$. Let $m=\max\{n_1,n_2, \dots,
n_k\}$. An easy induction argument shows that $\Orb(T,x) \subseteq
\span \{x, Tx, T^2x, \dots, T^{m-1}x\}$.
Hence the closed
linear span $\calN$ of $\Orb(T,x)$ is finite-dimensional. Observe that
the density in $\calM$ of $\Orb(T,x) \cap \calM$ implies that $\calM
\subseteq \calN$. Since clearly $\calN$ is an invariant subspace for $T$, Proposition
\ref{prop:restriction} shows that $T\restricted_\calN$ is
subspace-hypercyclic for $\calM$. But Theorem \ref{theo:finite_dim}
contradicts this, and so the proof is finished.
\end{proof}

For the following theorem, we assume the reader is familiar with some
basic facts about compact operators on Hilbert space, which can be
found in many standard references. A particularly nice exposition can
be found in \cite[p.~140--142]{RR2}.

\begin{theorem}
Let $T \in \BH$. If $T$ is compact, then $T$ is not
subspace-hypercyclic for any subspace.
\end{theorem}
\begin{proof}
Assume that $T$ is compact and subspace-hypercyclic for some
$\calL$. Since $T$ is compact, Theorem \ref{th:hyp_spec} implies
that there exists $\lambda$, an eigenvalue of $T$, such that
$\lambda \in S^1$. 

Since $T$ is compact there exists $N \in \N$ such that
$\ran(T-\lambda)^N=\ran(T-\lambda)^{N+k}$ for all $k \in \N$. Also, Proposition
\ref{prop:gen_kernel} gives that $\ker((T-\lambda)^*)^N \subseteq
\calL^\perp$ and hence $\calL \subseteq (\ker((T-\lambda)^*)^N)^\perp =
\ran(T-\lambda)^N$, since $\ran(T-\lambda)^N$ is closed. Since
$\calN:=\ker(T-\lambda)^N$ and $\calM:=\ran(T-\lambda)^N$ are
complementary and invariant under $T$, and since we showed that $\calL
\subseteq \calM$, we can apply Theorem \ref{theo:complementary} to
obtain that $T\restricted_\calM$ is subspace-hypercyclic for
$\calL$. Also, $\sigma(T\restricted_\calM)=\sigma(T)\setminus\{
\lambda\}$ and $T\restricted_\calM$ is a compact operator. 

The process above can be repeated starting with the compact operator
$T\restricted_\calM$. Since the original operator $T$ is compact,
it can have only a finite number of eigenvalues of modulus equal to
$1$. Hence, eventually, repeating this process leaves us with a
compact operator which is subspace-hypercyclic for $\calL$ but whose
spectrum does not intersect $S^1$, contradicting Theorem
\ref{th:hyp_spec}.
\end{proof}

We should mention that hyponormal operators (and hence normal
operators) cannot be subspace-hypercyclic since the norms of the
elements of the orbit are either decreasing or eventually increasing,
as proved by P.~Bourdon~\cite{Bourdon}.

\section{Some questions}

\begin{enumerate}

\item Let $T$ be an invertible operator. If $T$ is
  subspace-hypercyclic for some $\calM$, is $T^{-1}$
  subspace-hypercyclic for some space? If so, for which space?

\item If $T$ is subspace-hypercyclic for some $\calM$ and $\lambda$ is
  of modulus $1$, is $\lambda T$ subspace-hypercyclic for $\calM$?

\item If $T$ is hypercyclic, must there be a proper subspace $\calM$ such
  that $T$ is subspace-hypercylic for $T$?

\item Coanalytic Toeplitz operators can be subspace-hypercyclic (see
  Example~\ref{ex:toep}). Is there a classification of when this
  occurs (for example, {\it \'a la} Godefroy-Shapiro \cite{GoSh})?

\item Can Proposition~\ref{prop:gen_kernel} be generalized? That is,
  if $T$ is subspace-hypercyclic for some $\calM$ and $q$ is
  a complex polynomial, is it true that $\ker q(T^*) \subseteq
  \calM^\perp$? This question was suggested to us by the referee.

\end{enumerate}

\end{document}